\documentclass[11pt,a4paper]{article}

\usepackage[utf8]{inputenc}
\usepackage[T1]{fontenc}
\usepackage{amsmath,amssymb,amsthm}
\usepackage{mathtools}
\usepackage{geometry}
\usepackage{enumitem} 

\usepackage{hyperref}
\usepackage{algorithm}
\usepackage{algpseudocode}
\usepackage{tcolorbox}
\usepackage{natbib}

\usepackage{tikz}
\usetikzlibrary{arrows.meta,positioning}

\newtheorem{theorem}{Theorem}[section]
\newtheorem{lemma}[theorem]{Lemma}
\newtheorem{proposition}[theorem]{Proposition}
\newtheorem{corollary}[theorem]{Corollary}
\theoremstyle{definition}
\newtheorem{definition}[theorem]{Definition}
\newtheorem{remark}[theorem]{Remark}

\newtheorem{assumption}[theorem]{Assumption}

\newcommand{\tX}{\widetilde{X}}
\newcommand{\tU}{\widetilde{U}}

\newcommand{\Var}{\mathrm{Var}}
\newcommand{\Cov}{\mathrm{Cov}}
\newcommand{\E}{\mathbb{E}}
\newcommand{\Prob}{\mathbb{P}}
\newcommand{\R}{\mathbb{R}}

\newcommand{\F}{\mathcal{F}}
\newcommand{\G}{\mathcal{G}}
\newcommand{\B}{\mathcal{B}}
\newcommand{\Law}{\mathcal{L}}

\begin{document}

\title{\textbf{ Decoupling for Markov Chains}}

\author{
Nawaf Bou-Rabee$^{1,2}$ \and
Victor H.\ de la Pe\~na$^{3}$
}

\date{
$^{1}$Department of Mathematical Sciences, Rutgers University--Camden\\
$^{2}$Center for Computational Mathematics, Flatiron Institute, Simons Foundation\\
$^{3}$Department of Statistics, Columbia University
}

\maketitle

\begin{abstract}
Consider a Markov chain $(X_i)_{i\ge0}$ with invariant measure $\mu$
that admits the representation
$X_{i+1}=\Phi(X_i,U_i)$, where $(U_i)_{i\ge0}$ are i.i.d. random variables and $\Phi$ is a measurable  map.
We introduce a \emph{tangent--decoupled} process
$(\widetilde X_i)_{i\ge0}$ obtained by replacing $(U_i)$ with an
independent copy.  Conditional on the realized backbone $(X_i)$, the sequence
$(f(\widetilde X_i))$ is  independent. Although
$(\widetilde X_i)$ is not Markovian, under the same ergodicity assumptions that
ensure a law of large numbers for $(X_i)$, the empirical averages
$n^{-1}\sum_{i=1}^n f(\widetilde X_i)$ converge almost surely to $\mu(f)$.
In addition, for every
$f\in L^2(\mu)$ and every $N\ge1$,
\[
\Var\!\Bigl(\sum_{i=1}^N f(X_i)\Bigr)
\;\le\;
2\,\Var\!\Bigl(\sum_{i=1}^N f(\widetilde X_i)\Bigr),
\]
and therefore
$\sigma_f^2 \le 2\,\widetilde\sigma_f^{\,2}$ for the corresponding
time-average variance constants.
The inequality requires neither reversibility nor mixing assumptions. Its proof identifies the two sequences as tangent in the sense of decoupling theory and applies the sharp $L^2$ tangent decoupling inequality of de la Peña, Yao, and Alemayehu (2025).  
\end{abstract}
\section{Introduction}

Markov chain Monte Carlo (MCMC) methods are fundamental tools for approximating expectations under a target distribution $\mu$ when direct sampling is impractical; see, e.g.,  \citep{Liu,RobertCassella,Madras}. Mathematically, an MCMC method generates a Markov chain $(X_i)_{i \geq 0}$ on a measurable state space $(S, \mathcal{S})$ with transition kernel $P$ constructed such that $\mu$ is invariant, i.e., $\mu P = \mu$. For a measurable function $f: S \to \R$, the natural estimator of $\mu(f)$ is the ergodic average
\[
\bar{f}_N = \frac{1}{N}\sum_{i=1}^N f(X_i).
\]

Under standard regularity assumptions \citep{MeynTweedie1993,DoucMoulinesPriouretSoulier}, the estimator $\bar{f}_N$ converges almost surely to $\mu(f)$ and its fluctuations are asymptotically Gaussian:
\[
\sqrt{N}(\bar{f}_N - \mu(f)) \xrightarrow{d} \mathcal{N}(0, \sigma_f^2),
\]
where the time-average variance constant 
\[
\sigma_f^2 = \Var_\mu(f) + 2\sum_{k=1}^{\infty} \Cov_\mu(f(X_0), f(X_k))
\] governs the Monte Carlo error \cite[Chapter IV]{asmussen2007}.  

Accurately estimating or bounding $\sigma_f^2$ is a central problem in MCMC
error assessment.  Our approach introduces a \emph{tangent–decoupled} sequence that preserves the
stationary marginals of the original chain while greatly simplifying its
temporal dependence; see \cite{delapena-gine} for background on decoupling
methods.
This construction leads to two key results: the decoupled sequence is itself a
consistent estimator of $\mu(f)$, and comparison with the original chain
yields a sharp, uniform finite–sample variance inequality.
Together, these results provide a new and practical route to bounding
$\sigma_f^2$ for a broad class of modern MCMC algorithms.

\subsection{Representation via Auxiliary Randomness} \label{sec:aux-randomness}

Let $(\mathbb{A}, \B(\mathbb{A}))$ be a measurable space and let $(U_i)_{i \geq 0}$ be an i.i.d.\ sequence of $\mathbb{A}$-valued random variables with common distribution $\nu$. Suppose there exists a measurable update map
\[
\Phi: S \times \mathbb{A} \to S
\]
such that the Markov chain is generated by
\begin{equation}\label{eq:chain}
X_{i+1} = \Phi(X_i, U_i), \quad i \geq 0.
\end{equation}
The transition kernel is $P(x, A) = \Prob(\Phi(x, U_0) \in A)$.

\paragraph{Examples.}
Many common Markov chains used in statistics and machine learning admit a representation of the form 
\eqref{eq:chain}.  We list several illustrative examples.

\begin{enumerate}

\item \textbf{Gibbs Sampler (Random-Scan) \cite{CasellaGeorge1992,
Owen2013}.}  
Let $S=\mathbb{R}^d$ and $\mathbb{A}=\{1,\dots,d\}\times\mathbb{R}$ with $U=(j,W)$, 
where $j$ chooses a coordinate and $W$ is a draw from the conditional distribution of $X_j$ given the other coordinates.  
The update map is
\[
\Phi(x,(j,W)) = (x_1,\dots,x_{j-1}, W, x_{j+1},\dots,x_d).
\]

\item \textbf{Random-Walk Metropolis \cite{RobertsTweedie1996RWM,RobertsGelmanGilks1997}.}  
Let $S=\mathbb{R}^d$ and let $\mathbb{A} = \mathbb{R}^d \times [0,1]$ with $U=(Z,V)$, 
where $Z\sim q$ is a proposal increment and $V\sim \mathrm{Unif}(0,1)$.  
The map
\[
\Phi(x,(Z,V)) =
\begin{cases}
x+Z, & V \le \alpha(x, x+Z),\\
x,   & \text{otherwise},
\end{cases}
\]
encodes the entire Metropolis accept--reject step.

\item \textbf{Metropolis-Adjusted Langevin Algorithm (MALA) \cite{RobertsTweedie1996}.}  
Let $S=\mathbb{R}^d$ and $\mathbb{A}=\mathbb{R}^d \times [0,1]$ with $U=(Z,V)$, $Z\sim N(0,I)$, $V\sim\mathrm{Unif}(0,1)$.  
Given step size $h>0$, the proposal is
\[
Y = x - h \nabla U(x) + \sqrt{2h}\,Z,
\]
and the update map is
\[
\Phi(x,(Z,V)) =
\begin{cases}
Y, & V \le \alpha(x,Y),\\
x, & \text{otherwise}.
\end{cases}
\]

\item \textbf{Hamiltonian Monte Carlo (HMC) \cite{DuaneKennedyPendletonRoweth1987,Neal2011HMC}.}  
Let $S=\mathbb{R}^d$ and let $\mathbb{A}$ contain all auxiliary variables needed to simulate one HMC proposal: 
a momentum draw $p \sim N(0,M)$ and a sequence of Gaussian or uniform variables used in numerical integration 
and the acceptance decision.  
Let $U=(p,W)$, where $W$ is the uniform variable for Metropolis correction.  
Define
\[
\tilde{x} = \text{LeapfrogIntegrator}(x, p),
\]
and the update map is
\[
\Phi(x,(p,W)) = 
\begin{cases}
\tilde{x}, & W \le \alpha\big((x,p),(\tilde{x},\tilde{p})\big),\\
x, & \text{otherwise},
\end{cases}
\]

\item \textbf{No-U-Turn Sampler (NUTS) \cite{HoffmanGelman2014,CarpenterEtAl2017,BouRabeeCarpenterKleppeLiu2025}.}  
Let $S=\mathbb{R}^d$ and let $\mathbb{A}$ contain all auxiliary variables needed to simulate a full NUTS trajectory: 
a momentum draw $p\sim N(0,M)$, a Bernoulli sequence $\omega$ of left/right tree extensions, and the integration time $t$.  The NUTS update map
\[
\Phi(x, U) = \text{the point selected by NUTS driven by } U = (p,\omega,t)
\]
is measurable and produces the next state of the chain.  
All randomness used by NUTS (momentum, the random decisions in tree growth, and the integration time) is encapsulated in $U$, 
so NUTS also fits the form \eqref{eq:chain}.
\end{enumerate}

\subsection{Tangent-Decoupled Sequence}

Decoupling and tangency originate in the study of dependent sequences, where
one compares sums of adapted random variables with sums of conditionally
independent ``tangent’’ copies that share the same conditional distributions
given the past; see, for example, \cite{delapena-1994, delapena-gine}.  In that framework, a
companion sequence is constructed so that each term has the same one-step
conditional law as the original process, but with greatly simplified temporal
dependence.  Decoupling inequalities then allow one to transfer $L^p$ bounds
between the two sequences.

Motivated by this viewpoint, we now introduce a Markov–chain analogue in which
the companion process is obtained by running the same update map $\Phi$ along
the realized backbone $(X_i)$, but replacing the original auxiliary randomness
with an independent copy.  The construction is inspired by the classical
notion of tangency from decoupling theory,  but appears to be new in
the context of Markov chains generated through auxiliary randomness.

\begin{definition}[Tangent-decoupled sequence]\label{def:tangent}
Let $(X_i)_{i \geq 0}$ be generated by \eqref{eq:chain}. An independent i.i.d.\ copy $(\tU_i)_{i \geq 0}$ of $(U_i)$ is fixed. The tangent-decoupled sequence $(\tX_i)_{i \geq 0}$ is defined by
\begin{equation}\label{eq:decoupled}
\tX_{i+1} = \Phi(X_i, \tU_i), \quad i \geq 0.
\end{equation}
\end{definition}

Thus, the same update mechanism $\Phi$ is applied along the realized path of the original chain, but with fresh independent auxiliary variables at each step. Figure~\ref{fig:covariance-structure} illustrates the key structural difference
between the original Markov chain and the tangent-decoupled construction: in
the original chain, a single realization of the auxiliary randomness
propagates forward and induces long-range temporal dependence, whereas in the
tangent-decoupled chain each auxiliary variable affects only a single
transition.

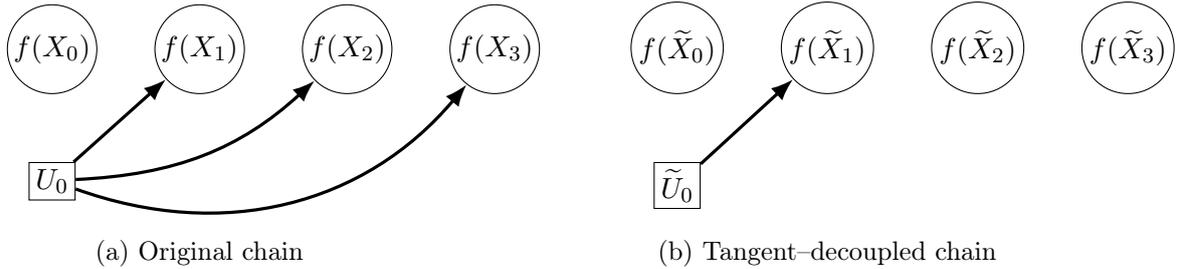
\begin{figure}[t]
\centering

\begin{minipage}[t]{0.48\textwidth}
\centering
\begin{tikzpicture}[
    >=Latex,
    node distance=0.75cm,
    obs/.style={circle, draw, inner sep=1pt, minimum size=13pt},
    noise/.style={rectangle, draw, inner sep=2.5pt, minimum size=11pt},
    covlabel/.style={font=\scriptsize},
    thickedge/.style={->, thick},
    mededge/.style={->, semithick},
    thinedge/.style={->, thin}
]

\node[obs] (X0) {$f(X_0)$};
\node[obs] (X1) [right=of X0] {$f(X_1)$};
\node[obs] (X2) [right=of X1] {$f(X_2)$};
\node[obs] (X3) [right=of X2] {$f(X_3)$};

\node[noise] (U0) [below=0.9cm of X0] {$U_0$};

\draw[very thick, ->] (U0) -- (X1);
\draw[very thick, ->] (U0) to[bend right=20] (X2);
\draw[very thick, ->] (U0) to[bend right=35] (X3);

\node[below=1.8cm of X1] {\small (a) Original chain};

\end{tikzpicture}
\end{minipage}
\hfill
\begin{minipage}[t]{0.48\textwidth}
\centering
\begin{tikzpicture}[
    >=Latex,
    node distance=0.75cm,
    obs/.style={circle, draw, inner sep=1pt, minimum size=13pt},
    noise/.style={rectangle, draw, inner sep=2.5pt, minimum size=11pt},
    covlabel/.style={font=\scriptsize},
    thickedge/.style={->, thick},
    mededge/.style={->, semithick},
    thinedge/.style={->, thin}
]

\node[obs] (Xt0) {$f(\widetilde X_0)$};
\node[obs] (Xt1) [right=of Xt0] {$f(\widetilde X_1)$};
\node[obs] (Xt2) [right=of Xt1] {$f(\widetilde X_2)$};
\node[obs] (Xt3) [right=of Xt2] {$f(\widetilde X_3)$};

\node[noise] (Ut0) [below=0.9cm of Xt0] {$\widetilde U_0$};


\draw[very thick, ->] (Ut0) -- (Xt1);

\node[below=1.8cm of Xt1] {\small (b) Tangent--decoupled chain};

\end{tikzpicture}
\end{minipage}

\caption{
Propagation of auxiliary randomness in the original chain (panel~(a)) versus the
tangent--decoupled chain (panel~(b)).  
In the original chain, $U_0$ influences all downstream states.  
In the tangent--decoupled chain, $\widetilde U_0$ affects only one transition.  
}
\label{fig:covariance-structure}
\end{figure}

\subsection{Main results}

This work develops a general theory of \emph{decoupling} for Markov
chains generated by i.i.d.\ auxiliary randomness.  Our contributions have two
parts.  First, we show that the tangent--decoupled sequence preserves the
ergodic mean of the original chain and therefore yields a valid estimator of
$\mu(f)$.  Second, we prove a sharp nonasymptotic variance comparison that
holds uniformly in the sample size and simultaneously for all observables in
$L^2(\mu)$.

\begin{theorem}[Consistency of the tangent--decoupled estimator]
\label{thm:intro-consistency}
Let $(X_i)_{i\ge 0}$ be a stationary and ergodic Markov chain generated by
\eqref{eq:chain} with invariant distribution $\mu$, and let
$(\widetilde X_i)_{i\ge 1}$ be its tangent--decoupled companion process defined
in \eqref{eq:decoupled}.  Assume that $X_0\sim\mu$.
If $f\in L^1(\mu)$ and $f\circ\Phi\in L^2(\mu\times\nu)$, then
\[
\frac{1}{n}\sum_{i=1}^n f(\widetilde X_i)
\;\xrightarrow{\mathrm{a.s.}}\;
\mu(f).
\]
\end{theorem}

The tangent-decoupled sequence is therefore not merely an auxiliary
construction: it yields a consistent Monte Carlo estimator of $\mu(f)$.
This result is noteworthy for several reasons.  The sequence $(\widetilde X_i)$
is not a Markov chain and is not generated by the transition kernel $P$; it is
obtained by replacing the auxiliary randomness driving the original chain with
an independent copy.  Such a perturbation typically destroys the long--run
behavior of the chain, and there is no \emph{a priori} reason to expect
empirical averages based on $(\widetilde X_i)$ to converge to the correct
limit.

Nevertheless, Theorem~\ref{thm:intro-consistency} shows that the
tangent-decoupled sequence remains ergodically correct \emph{in the sense of
empirical averages}: for every integrable observable $f$, its empirical mean
converges almost surely to $\mu(f)$, provided the backbone chain is stationary
and ergodic and mild integrability conditions hold.  Thus $(\widetilde X_i)$
provides a second, structurally different Monte Carlo estimator of $\mu(f)$;
one whose dependence structure is much simpler, since all stochastic temporal
dependence arising from the auxiliary randomness has been removed.

This structural disentangling of the auxiliary randomness is what enables the
second main result: a uniform finite--sample comparison of variances.

\begin{theorem}[Finite--sample and asymptotic variance comparison]
\label{thm:variance_intro}
Let $(X_i)_{i\ge0}$ be generated by \eqref{eq:chain} and let
$(\widetilde X_i)_{i\ge0}$ be its tangent-decoupled companion sequence defined
in \eqref{eq:decoupled}.  Then for every $f\in L^2(\mu)$ and every sample size $N\ge1$,
\begin{equation}\label{eq:finite-sample-variance}
\Var\!\left(\sum_{i=1}^N f(X_i)\right)
\;\le\;
2\,
\Var\!\left(\sum_{i=1}^N f(\widetilde X_i)\right).
\end{equation}
Consequently, the corresponding time-average variance constants satisfy
\begin{equation}\label{eq:asymptotic-variance}
\sigma_f^2 \;\le\; 2\,\widetilde\sigma_f^{\,2},
\qquad f\in L^2(\mu).
\end{equation}
\end{theorem}

The inequality \eqref{eq:finite-sample-variance} is noteworthy for several
reasons.  First, it is \emph{uniform in}~$N$: the comparison holds for every
finite sample size, without any asymptotic limiting argument.  Second, it
applies \emph{simultaneously to all} $f\in L^2(\mu)$, without regularity
conditions such as boundedness, Lipschitz continuity, or moment assumptions
beyond $L^2$.  Third, it requires neither reversibility, spectral structure,
nor mixing conditions; the auxiliary--variable representation of the transition
mechanism alone suffices.  Finally, the constant $2$ is best possible: for the
AR(1) chain with $a=\tfrac12$ and $f(x)=x$, equality holds
($\sigma_f^2 = 2\,\widetilde\sigma_f^{\,2}$); see Section~\ref{sec:ar1}.
To our knowledge, this is the first general variance comparison of this form
for Markov chains generated by i.i.d.\ auxiliary randomness.

As a consequence, for any MCMC algorithm of the form \eqref{eq:chain} that
admits a central limit theorem, the bound \eqref{eq:asymptotic-variance}
yields computable and provably conservative confidence intervals for
$\mu(f)$.  This is developed in Section~\ref{sec:CI}.  In particular, one may
estimate the asymptotic variance of the tangent--decoupled sequence and use
$\sigma_f^2 \le 2\,\widetilde\sigma_f^{\,2}$ to produce a valid upper bound for
the asymptotic variance of the original chain, without relying on spectral
methods, autocorrelation truncation, or mixing--rate assumptions.  This leads
to a simple and robust diagnostic for quantifying Monte Carlo uncertainty
across a wide range of modern samplers, including Gibbs updates, Langevin
algorithms, HMC, and NUTS.

The effectiveness of this approach is demonstrated through empirical studies in
Section~\ref{sec:ar1} and Section~\ref{sec:mcmc}, where we examine both the
variance comparison and the behavior of the tangent–decoupled estimator itself.
Across all examples, the finite–sample inequality 
\eqref{eq:finite-sample-variance} holds for both linear and nonlinear
observables, and the resulting confidence intervals are consistently
conservative yet informative, even in regimes where traditional HAC estimators
are unstable.  We also verify numerically the almost–sure consistency asserted
by Theorem~\ref{thm:intro-consistency}: empirical averages based on the
tangent–decoupled sequence converge reliably to the correct mean $\mu(f)$,
illustrating that the estimator is not only theoretically justified but also
stable in practice.  Taken together, these experiments show that tangent
decoupling provides a broadly applicable and practically robust tool for
quantifying Monte Carlo uncertainty across a wide range of algorithms,
dimensions, and correlation structures.

\section{Tangent Decoupling for Markov Chains}

The goal of this section is to formalize and analyze the tangent–decoupled 
construction introduced in the introduction.  Although the sequence 
$(\widetilde X_i)$ is neither Markovian nor generated by the transition kernel 
$P$, it enjoys two structural properties that make it surprisingly tractable: 
(i) given the backbone path $(X_i)$, the variables $(\widetilde X_i)$ become 
conditionally independent, and (ii) each $\widetilde X_i$ has the correct 
stationary marginal distribution.  These two observations place the 
tangent–decoupled process in a regime where classical decoupling ideas can be 
applied, and they are the ingredients that ultimately lead to the consistency 
and variance comparison theorems stated in Section~1.  We now establish these 
properties in turn.

\subsection{Conditional Independence Given the Backbone}

A key property of the tangent--decoupled sequence is that,
although $(\widetilde X_i)$ is not Markovian and does not evolve according to
the kernel $P$, its randomness is ``conditionally fresh’’ at each step once the
realized backbone path $(X_i)$ is fixed.  
All dependence between the $\widetilde X_i$ flows only through their common
conditioning on the backbone states $X_0, X_1,\ldots$. This mirrors the classical idea of \emph{decoupling} in probability: after
conditioning on a suitable $\sigma$--algebra, the ``tangent’’ sequence becomes
independent across time.  

The lemma below formalizes this conditional independence property.

\begin{lemma}[Conditional independence]\label{lem:cond-indep}
Let $(X_i)$ and $(\tX_i)$ be as in \eqref{eq:chain}--\eqref{eq:decoupled}, and let $\G := \sigma(X_i: i \geq 0)$ be the $\sigma$-algebra generated by the backbone chain $(X_i)$. Then the sequence $(\tX_i)_{i \geq 1}$ is conditionally independent given $\G$.  
\end{lemma}

\begin{proof}
For each $k\ge1$,
\[
\widetilde X_k = \Phi(X_{k-1},\widetilde U_{k-1}),
\]
where $(\widetilde U_i)$ is i.i.d.\ and independent of $X_0$ and $(U_i)$.

Conditional on $\mathcal{G}$:

\begin{itemize}
    \item the backbone values $X_0, X_1, \ldots$ are fixed;
    \item the maps $u \mapsto \Phi(X_{k-1},u)$ are deterministic;
    \item the auxiliary variables $\widetilde U_0,\widetilde U_1,\ldots$
          remain independent, since they are independent of $\mathcal{G}$.
\end{itemize}

Thus each $\widetilde X_k$ is a measurable function of an independent variable
$\widetilde U_{k-1}$, and hence the family $(\widetilde X_k)_{k\ge1}$ is
conditionally independent given $\mathcal{G}$.
\end{proof}
\subsection{Stationary Marginals}

In this part, we assume that $P$ admits a stationary distribution $\mu$ and that the chain is started in stationarity, so that $X_0 \sim \mu$ and hence $X_i \sim \mu$ for all $i \geq 0$.  
The next lemma shows that stationarity is preserved exactly: every
$\widetilde X_i$ has distribution $\mu$.  

\begin{lemma}[Stationary marginals under tangent decoupling]\label{lem:marginals}
Let $(\tX_i)$ be defined by \eqref{eq:decoupled}. Then, for every $i \geq 1$,
\[
\tX_i \sim \mu.
\]
\end{lemma}

\begin{proof}
Fix $i \geq 0$ and let $g: S \to \R$ be bounded and measurable. Conditioning on $X_i$ and using independence of $\tU_i$,
\[
\E[g(\tX_{i+1}) | X_i] = \int_S g(y)\, P(X_i, dy).
\]
Since $X_i \sim \mu$,
\[
\E[g(\tX_{i+1})] = \int_S \int_S g(y)\, P(x, dy)\, \mu(dx) = \int_S g(y)\, \mu(dy),
\]
the last equality being the stationarity relation $\mu P = \mu$. Thus $\tX_{i+1} \sim \mu$.
\end{proof}

\subsection{Consistency of the Tangent--Decoupled Estimator}

A natural question raised by Lemmas~\ref{lem:cond-indep} and
\ref{lem:marginals} is whether the tangent--decoupled sequence $(\tX_i)$,
which is neither Markovian nor generated by the original transition kernel,
can nonetheless be used to estimate expectations with respect to the target
distribution $\mu$.  The key observation is that $(\tX_i)$ has the correct
stationary marginals (Lemma~\ref{lem:marginals}) and exhibits conditional
independence given the backbone (Lemma~\ref{lem:cond-indep}), a structure that
is fundamentally different from the dependence present in the original chain.

The following theorem shows that, under the same mild conditions ensuring a
law of large numbers for $(X_i)$, the empirical mean based on $(\tX_i)$
converges almost surely to $\mu(f)$ for every integrable observable $f$.
Thus the tangent--decoupled sequence provides an estimator with the
\emph{same} asymptotic mean as the original chain but with a  simplified
dependence structure.

\begin{theorem}[Consistency of the tangent--decoupled estimator]
\label{thm:decoupled-consistency}
Let $(X_i)_{i\ge 0}$ be a stationary and ergodic Markov chain with invariant
distribution $\mu$, defined by \eqref{eq:chain}, and let
$(\widetilde X_i)_{i\ge 0}$ be the associated tangent--decoupled sequence
defined by \eqref{eq:decoupled}, with $X_0\sim\mu$.
Let $f:S\to\R$ satisfy $f\in L^1(\mu)$ and
$f\circ\Phi \in L^2(\mu\times\nu)$.
Then
\[
\frac{1}{n}\sum_{i=1}^n f(\tX_i)
\;\xrightarrow{\textnormal{a.s.}}\;
\mu(f).
\]
\end{theorem}

\begin{proof}
Define the conditional expectation
\[
g(x) := \E[f(\Phi(x,U_0))],
\]
which is finite $\mu$-a.e.\ because $f\circ\Phi\in L^2(\mu\times\nu)$ (hence in $L^1(\mu\times\nu)$ since $\mu\times\nu$ is a probability measure).
We decompose the tangent--decoupled estimator into a ``backbone term''
and a ``fluctuation term'':
\[
\frac{1}{n}\sum_{i=1}^n f(\tX_i)
= \underbrace{\frac{1}{n}\sum_{i=1}^n g(X_{i-1})}_{A_n}
  \;+\;
  \underbrace{\frac{1}{n}\sum_{i=1}^n \bigl(f(\tX_i)-g(X_{i-1})\bigr)}_{B_n}.
\]

We treat the two terms separately.

\paragraph{Step 1: Convergence of \(A_n\).}
Since $\mu$ is stationary,
\[
\mu(g)
= \int_S g(x)\,\mu(dx)
= \int_S \int_U f(\Phi(x,u))\,\nu(du)\,\mu(dx)
= \int_S f(y)\,\mu(dy)
= \mu(f),
\]
using the invariance identity $\mu P=\mu$.  
Because $g\in L^1(\mu)$ and $(X_i)$ is ergodic, the ergodic theorem gives
\[
A_n = \frac{1}{n}\sum_{i=1}^n g(X_i)
\;\xrightarrow{\text{a.s.}}\;
\mu(g)=\mu(f).
\]

\paragraph{Step 2: Convergence of \(B_n\).}
Define
\[
Y_i := f(\tX_i) - g(X_{i-1}).
\]
We will show that $\frac{1}{n}\sum_{i=1}^n Y_i\to0$ a.s.

First, by Lemma~\ref{lem:cond-indep}, the sequence $(\tX_i)$ is
conditionally independent given $\G := \sigma(X_i:i\ge0)$.
Therefore, the sequence $(Y_i)$ is also conditionally independent given $\G$.

Second, for each $i$,
\[
\E[Y_i\mid X_{i-1}]
= \E\!\left[f\!\bigl(\Phi(X_{i-1},\tU_{i-1})\bigr)\,\middle|\, X_{i-1}\right]
  - g(X_{i-1})
= 0.
\]
Thus $(Y_i)$ is a sequence of conditionally independent,
conditionally mean-zero random variables.

Third, we compute conditional variances.
Since $f\circ\Phi\in L^2(\mu\times\nu)$,
\[
\E[Y_i^2]
= \E[(f(\Phi(X_i,\tU_i)) - g(X_i))^2]
< \infty.
\]
Define
\[
V_i := \Var(Y_i \mid X_i).
\]
Because $Y_i$ is centered given $X_i$,
\[
V_i = \E[Y_i^2 \mid X_i] \le \E[Y_i^2].
\]
Stationarity of $(X_i,\tU_i)$ implies $\E[Y_i^2]=\E[Y_0^2]$ for all $i$.

Consider the series
\[
S := \sum_{i=1}^\infty \frac{V_i}{i^2}.
\]
We have
\[
\E[S]
= \sum_{i=1}^\infty \frac{\E[V_i]}{i^2}
\;\le\; \E[Y_0^2]\sum_{i=1}^\infty\frac{1}{i^2}
<\infty.
\]
Since $S$ is nonnegative, $\E[S]<\infty$ implies
$S<\infty$ almost surely.

Now condition on $\G$.  Given $\G$, the variables $(Y_i)$ are independent,
mean-zero, and satisfy
\[
\sum_{i=1}^\infty \frac{\Var(Y_i\mid\G)}{i^2}
=\sum_{i=1}^\infty \frac{V_i}{i^2}
<\infty
\quad\text{a.s.}
\]
Thus Kolmogorov's strong law for independent, not necessarily
identically distributed, mean-zero variables yields
\[
B_n
= \frac{1}{n}\sum_{i=1}^n Y_i
\;\xrightarrow{\text{a.s.}}\;
0.
\]

\paragraph{Step 3: Combine the limits.}
We have shown
\[
A_n\to\mu(f)\quad\text{a.s.},
\qquad
B_n\to 0\quad\text{a.s.}
\]
and therefore
\[
\frac{1}{n}\sum_{i=1}^n f(\tX_i)
= A_n + B_n
\;\xrightarrow{\text{a.s.}}\;
\mu(f).
\]
\end{proof}

\begin{remark}[Effect of averaging multiple tangent draws]
\label{rem:multi-tangent-variance}
A natural extension of the tangent--decoupled estimator is to generate, for each
backbone state $X_i$, multiple independent tangent draws
\[
\tX_{i,1},\ldots,\tX_{i,M},
\qquad
\tX_{i,m} := \Phi(X_i,\tU_{i,m}), \quad \tU_{i,m}\stackrel{\mathrm{i.i.d.}}{\sim}\nu,
\]
and to consider the averaged estimator
\[
\hat\mu_{n,M}(f)
:= \frac{1}{nM}\sum_{i=1}^n\sum_{m=1}^M f(\tX_{i,m}).
\]
By conditional independence, this estimator remains unbiased and consistent for
$\mu(f)$.

As $M\to\infty$, the auxiliary randomness is averaged out and
\[
\hat\mu_{n,M}(f)
\;\longrightarrow\;
\frac{1}{n}\sum_{i=1}^n g(X_i),
\qquad
g(x):=\E[f(\Phi(x,U))]=(Pf)(x),
\]
almost surely for fixed $n$.  
In particular, the estimator remains consistent in the limit $M\to\infty$, since
$\mu(g)=\mu(f)$ by invariance of $\mu$ under $P$. Consequently, increasing $M$ does not yield an estimator whose variance approaches
$\Var_\mu(f)$; rather, it produces a Rao--Blackwellized estimator based on the
backbone chain with observable $g=Pf$.

If a central limit theorem holds for additive functionals of $(X_i)$, the
corresponding asymptotic variance is
\[
\sigma_g^2
=
\Var_\mu(g(X_0))
+2\sum_{k\ge1}\Cov_\mu\bigl(g(X_0),g(X_k)\bigr),
\]
which in general differs from the asymptotic variance
\[
\sigma_f^2
=
\Var_\mu(f(X_0))
+2\sum_{k\ge1}\Cov_\mu\bigl(f(X_0),f(X_k)\bigr)
\]
of the standard Markov chain estimator.
In particular, even in the limit $M\to\infty$, the variance is governed by the
temporal dependence of the backbone chain and not by the marginal variance
$\Var_\mu(f)$.

When the backbone kernel $P$ is reversible, the transformation $f\mapsto Pf$
attenuates high-frequency spectral components while leaving slow modes essentially
unchanged.  Thus averaging over tangent draws removes variance due to auxiliary
randomness but does not eliminate the intrinsic correlation structure of the backbone
chain.
\end{remark}

\subsection{Tangency and Its Preservation}

Tangency is the property that links the original Markov chain
$(X_i)$ with its tangent--decoupled companion $(\widetilde X_i)$.  In the
language of decoupling theory \cite{delapena-gine}, two sequences are
\emph{tangent} if they share the same conditional laws given the past.
This property is the key to transferring $L^2$ bounds from the conditionally
independent sequence $(\widetilde X_i)$ to the dependent sequence $(X_i)$.

\begin{definition}[Tangent sequences]
Let $(\mathcal F_i)_{i\ge 0}$ be a filtration.  
Two sequences $(d_i)_{i \geq 1}$ and $(e_i)_{i \geq 1}$ are $\F_{i-1}$-tangent if for each $i \geq 1$:
\[
\Law(d_i | \F_{i-1}) = \Law(e_i | \F_{i-1}) \quad \text{a.s.}
\] where $\mathcal L(\,\cdot\,|\,\mathcal F_{i-1})$ denotes the conditional distribution
given $\mathcal F_{i-1}$.
\end{definition}

\begin{lemma}[Tangency of original and decoupled chains]\label{lem:tangency}
Let $\F_{i-1}$ be the filtration defined by \[
\F_{i-1} = \sigma(X_0, X_1, \ldots, X_{i-1};~\tX_1, \ldots, \tX_{i-1}) \;.
\] The sequences $(X_i)_{i \geq 1}$ and $(\tX_i)_{i \geq 1}$ are $\F_{i-1}$-tangent.
\end{lemma}

\begin{proof}
For each $i\ge1$,
\[
X_i=\Phi(X_{i-1},U_{i-1}),
\qquad
\widetilde X_i=\Phi(X_{i-1},\widetilde U_{i-1}),
\]
where $U_{i-1}$ and $\widetilde U_{i-1}$ are independent of
$\mathcal F_{i-1}$ and identically distributed with law $\nu$.
Hence
\[
\Law(X_i\mid\mathcal F_{i-1})
=
P(X_{i-1},\cdot)
=
\Law(\widetilde X_i\mid\mathcal F_{i-1}).
\]
\end{proof}

The next lemma shows that tangency is stable under arbitrary measurable
transformations.  This is crucial because the variance comparison in Section~\ref{sec:varcomp} is performed not on the states themselves but on transformed variables such as $f(X_i)$.

\begin{lemma}[Tangency preserved under measurable transformations]\label{lem:tangency-preserved}
Let $(d_i)_{i \geq 1}$ and $(e_i)_{i \geq 1}$ be sequences of random variables taking values in a measurable space $(S, \mathcal{S})$, and let $(\F_i)_{i \geq 0}$ be a filtration. Suppose $(d_i)$ and $(e_i)$ are $\F_{i-1}$-tangent. Then for any measurable function $f: S \to \R$:
\begin{enumerate}[label=(\alph*)]
\item The sequences $(f(d_i))$ and $(f(e_i))$ are $\F_{i-1}$-tangent.
\item If $(e_i)$ is conditionally independent given a $\sigma$-algebra $\G$, then $(f(e_i))$ is also conditionally independent given $\G$.
\end{enumerate}
\end{lemma}

\begin{proof}

\noindent\textbf{Part (a).}
Fix $i \ge 1$. Let $\kappa_d(\omega,\cdot)$ and $\kappa_e(\omega,\cdot)$ denote
regular conditional distributions of $d_i$ and $e_i$ given $\mathcal F_{i-1}$.
By $\mathcal F_{i-1}$--tangency, there exists a null set $N_0$ such that for all
$\omega\notin N_0$,
\[
\kappa_d(\omega,A)=\kappa_e(\omega,A)
\qquad\text{for all }A\in\mathcal S.
\]

The conditional distribution of $f(d_i)$ given $\mathcal F_{i-1}$ is the
pushforward measure $\kappa_d(\omega,\cdot)\circ f^{-1}$, and similarly for
$f(e_i)$. To establish tangency, it therefore suffices to show that these
pushforward measures agree almost surely.

Let
\[
\mathcal C := \{(-\infty,q]: q\in\mathbb Q\},
\]
a countable $\pi$--system generating $\mathcal B(\mathbb R)$. For any
$B\in\mathcal C$ and $\omega\notin N_0$, measurability of $f$ implies
$f^{-1}(B)\in\mathcal S$, and hence
\begin{equation}\label{eq:tangency}
(\kappa_d(\omega,\cdot)\circ f^{-1})(B)
= \kappa_d(\omega,f^{-1}(B))
= \kappa_e(\omega,f^{-1}(B))
= (\kappa_e(\omega,\cdot)\circ f^{-1})(B).
\end{equation}

For fixed $\omega\notin N_0$, define
\[
\mathcal D_\omega
:= \bigl\{B\in\mathcal B(\mathbb R):
(\kappa_d(\omega,\cdot)\circ f^{-1})(B)
= (\kappa_e(\omega,\cdot)\circ f^{-1})(B)\bigr\}.
\]
Then:
\begin{itemize}
\item $\mathcal D_\omega$ is a $\lambda$-system (closed under complements and countable disjoint unions);
\item $\mathcal C\subseteq\mathcal D_\omega$ by \eqref{eq:tangency};
\item $\mathcal C$ is a $\pi$--system generating $\mathcal B(\mathbb R)$.
\end{itemize}
By the $\pi$--$\lambda$ theorem,
\[
\mathcal D_\omega=\mathcal B(\mathbb R).
\]
Thus, for all $\omega\notin N_0$,
\[
\Law(f(d_i)\mid\mathcal F_{i-1})(\omega)
=
\Law(f(e_i)\mid\mathcal F_{i-1})(\omega),
\]
and hence $(f(d_i))$ and $(f(e_i))$ are $\mathcal F_{i-1}$--tangent.

\medskip
\textbf{Part (b).}
Suppose $(e_i)$ is conditionally independent given $\mathcal G$. By
definition, for all $n\ge1$ and all bounded measurable functions
$g_1,\dots,g_n:S\to\mathbb R$,
\[
\mathbb E\!\left[\prod_{k=1}^n g_k(e_k)\,\middle|\,\mathcal G\right]
=
\prod_{k=1}^n \mathbb E[g_k(e_k)\mid\mathcal G]
\quad\text{a.s.}
\]

Let $h_1,\dots,h_n:\mathbb R\to\mathbb R$ be bounded measurable functions and
define $g_k:=h_k\circ f$, which are bounded measurable functions on $S$.
Applying the above identity yields
\[
\mathbb E\!\left[\prod_{k=1}^n h_k(f(e_k))\,\middle|\,\mathcal G\right]
=
\prod_{k=1}^n \mathbb E[h_k(f(e_k))\mid\mathcal G]
\quad\text{a.s.}
\]
Since this holds for all $n\ge1$ and all bounded measurable
$h_1,\dots,h_n$, it follows that $(f(e_i))$ is conditionally independent
given $\mathcal G$.

\end{proof}
\section{Variance Comparison} \label{sec:varcomp}

For $f \in L^2(\mu)$, define the centered sequences
\[
d_i = f(X_i) - \mu(f), 
\qquad 
e_i = f(\widetilde X_i) - \mu(f) \;.
\]
Assume the following limits exist and are finite:
\[
\sigma_f^2 
= \lim_{N \to \infty} N\,\Var\!\left(\frac{1}{N}\sum_{i=1}^N d_i\right),
\qquad
\widetilde\sigma_f^{\,2} 
= \lim_{N \to \infty} N\,\Var\!\left(\frac{1}{N}\sum_{i=1}^N e_i\right).
\]

\begin{theorem}[Finite-sample and asymptotic variance domination]
\label{thm:variance}
Let $(X_i)_{i\ge 0}$ be a Markov chain of the form~\eqref{eq:chain},
and let $(\widetilde X_i)_{i\ge 0}$ be its tangent-decoupled sequence
\eqref{eq:decoupled}. Then for every $f \in L^2(\mu)$ and every $N \ge 1$,
\begin{equation}\label{eq:finite-var-dom}
\Var\!\left(\sum_{i=1}^N d_i\right)
\;\le\; 
2\,\Var\!\left(\sum_{i=1}^N e_i\right).
\end{equation}
Consequently, $
\sigma_f^2 \;\le\; 2\,\widetilde\sigma_f^{\,2}$.
\end{theorem}

The proof is an application of the sharp $L^2$ tangent decoupling inequality of de la Peña, Yao, and Alemayehu (2025) \cite{delapena-yao-alemayehu}, after verifying that $(d_i)$ and $(e_i)$ satisfy the required tangency and conditional independence conditions.

\begin{remark}[Sharpness of the constant]
The factor $2$ in Theorem~\ref{thm:variance} cannot be improved in
general.  Section~\ref{sec:ar1} shows that for the stationary AR(1) chain
\[
X_{i+1} = a X_i + \xi_i, \qquad |a|<1,
\]
with observable $f(x)=x$, the ratio $\sigma_f^2 / \widetilde\sigma_f^{\,2}$ is equal to $2$ when $a=\tfrac12$.  Thus equality is attained, and the constant
in the variance domination inequality is optimal.
\end{remark}

\begin{proof}
By Lemmas~\ref{lem:tangency} and~\ref{lem:tangency-preserved},
the sequences $(d_i)$ and $(e_i)$ are $\mathcal F_{i-1}$-tangent,
and $(e_i)$ is conditionally independent given
$\mathcal G = \sigma(X_i : i \ge 0)$.

By the $L^2$ decoupling inequality of de la Pe\~na, Yao and Alemayehu (2025) \cite{delapena-yao-alemayehu},
for all $N \ge 1$ we have
\[
\mathbb{E}\!\left[\Bigl(\sum_{i=1}^N d_i\Bigr)^2\right]
\;\le\; 2\,
\mathbb{E}\!\left[\Bigl(\sum_{i=1}^N e_i\Bigr)^2\right].
\]
Moreover, tangency implies
$\mathbb{E}[d_i] = \mathbb{E}[e_i]$ for each $i$ (by equality of conditional expectations under tangency), hence
\[
\Var\!\left(\sum_{i=1}^N d_i\right)
= \mathbb{E}\!\left[\Bigl(\sum_{i=1}^N d_i\Bigr)^2\right]
      - \Bigl(\mathbb{E}\Bigl[\sum_{i=1}^N d_i\Bigr]\Bigr)^2,
\]
and similarly for $(e_i)$. Since the expectations of the sums coincide,
the same inequality holds for the variances:
\[
\Var\!\left(\sum_{i=1}^N d_i\right)
\;\le\; 2\,\Var\!\left(\sum_{i=1}^N e_i\right),
\]
which is \eqref{eq:finite-var-dom}.

Dividing both sides by $N$ yields
\[
 \Var\!\left(\frac{1}{\sqrt{N}}\sum_{i=1}^N d_i\right)
\;\le\;
2\, \Var\!\left(\frac{1}{\sqrt{N}}\sum_{i=1}^N e_i\right).
\]
Letting $N\to\infty$ gives $\sigma_f^2 \le 2\widetilde\sigma_f^{\,2}$, as required.
\end{proof}


\section{Consequences for confidence bounds} \label{sec:CI}

This section shows how the variance domination in Theorem~\ref{thm:variance}
can be used to construct practical and conservative confidence bounds for
$\mu(f)$. 

For $N \ge 1$ define the empirical means
\[
\bar f_N := \frac{1}{N}\sum_{i=1}^N f(X_i),
\qquad
\bar f_N^{\,\widetilde X} := \frac{1}{N}\sum_{i=1}^N f(\widetilde X_i),
\]
so that
\[
\sum_{i=1}^N d_i
  = N\bigl(\bar f_N - \mu(f)\bigr),
\qquad
\sum_{i=1}^N e_i
  = N\bigl(\bar f_N^{\,\widetilde X} - \mu(f)\bigr),
\]
where $d_i := f(X_i)-\mu(f)$ and $e_i := f(\widetilde X_i)-\mu(f)$.
The finite–sample variance comparison of 
Theorem~\ref{thm:variance} gives
\begin{equation}\label{eq:var-mean-finite}
\Var(\bar f_N)
= \frac{1}{N^2}\Var\!\Bigl(\sum_{i=1}^N d_i\Bigr)
\;\le\;
\frac{2}{N^2}\Var\!\Bigl(\sum_{i=1}^N e_i\Bigr)
= 2\,\Var(\bar f_N^{\,\widetilde X}).
\end{equation}

To make use of this bound, we require an estimator of the decoupled long–run
variance.  This can be obtained using standard long–run variance estimators,
such as the Bartlett or Newey–West HAC estimators
\cite{Bartlett1946,NeweyWest1987}.  These estimators are widely used for
dependent data and have also been employed in the MCMC context; see, for
example, \cite{FlegalJones2010,Atchade2011,FlegalJones2011} for discussion and
theoretical guarantees. In addition, since many MCMC methods used in probabilistic programming languages (including NUTS) are reversible, it is natural to employ Geyer’s initial positive sequence (IPS) estimator \cite{geyer1992}. The IPS estimator exploits a structural positivity property of the autocorrelation function that holds for reversible Markov chains, yielding a stable and fully automatic truncation rule for estimating the long–run variance. As shown in Section~\ref{sec:mcmc}, this positivity property is inherited by the tangent–decoupled sequence, making Geyer’s method particularly well suited for uncertainty quantification in the present setting.

\begin{assumption}[Consistent estimator of the decoupled variance]
\label{ass:decoupled-consistent}
There exists an estimator $\widetilde{\sigma}_{f,N}^2$ such that
\[
    \widetilde{\sigma}_{f,N}^2 \xrightarrow{P} \widetilde{\sigma}_f^{\,2}
    \qquad (N\to\infty),
\]
where $\widetilde{\sigma}_f^{\,2}$ is  the long-run (asymptotic) variance of $\bar f_N^{\,\widetilde X}$.
We additionally assume that $\widetilde{\sigma}_f^{\,2}>0$.
\end{assumption}

Assumption~\ref{ass:decoupled-consistent} is mild: it holds whenever
$(e_i)$ admits a long--run variance and satisfies a standard CLT (for
example, under geometric ergodicity of $(X_i)$ together with
$f \in L^{2}(\mu)$).

We now obtain a conservative confidence interval for $\mu(f)$.

\begin{corollary}[Confidence intervals under CLT]\label{cor:CI}
Suppose the assumptions of Theorem~\ref{thm:variance} and
Assumption~\ref{ass:decoupled-consistent} hold, and assume that
\[
\sqrt{N}\bigl(\bar f_N - \mu(f)\bigr)
\;\xrightarrow{d}\; N(0,\sigma_f^2)
\qquad (N\to\infty),
\]
where $\sigma_f^2 = \lim_{N\to\infty} N\,\Var(\bar f_N)$.
Then for any $\alpha\in(0,1)$ the interval
\[
\mathrm{CI}_N(f)
  := \bar f_N
     \pm
     z_{1-\alpha/2}\,
     \sqrt{\frac{2\,\widetilde\sigma_{f,N}^2}{N}}
\]
is an asymptotically conservative $(1-\alpha)$ confidence interval for
$\mu(f)$, i.e.
\[
\liminf_{N\to\infty}
\Prob\bigl(\,\mu(f) \in \mathrm{CI}_N(f)\,\bigr)
\;\ge\; 1-\alpha.
\]
\end{corollary}

Before turning to the proof, it is useful to summarize the structure of the
argument.  
The variance comparison from Theorem~\ref{thm:variance} implies that the true
asymptotic variance $\sigma_f^2$ of the original empirical mean is no larger than twice that of the
tangent–decoupled sequence.  
When combined with the CLT for $\bar f_N$ and the consistency of the variance estimator
$\widetilde\sigma_{f,N}^2$, this leads to a standardized statistic whose
limiting variance does not exceed one.  
The resulting bound guarantees that the normal quantile yields conservative
long–run coverage.

\begin{proof}
By Theorem~\ref{thm:variance}, 
\[
\sigma_f^2 \;\le\; 2\,\widetilde\sigma_f^{\,2}.
\]
Let
\[
T_N := \frac{\sqrt{N}\,(\bar f_N - \mu(f))}{\sqrt{2\,\widetilde\sigma_{f,N}^2}}.
\]

We analyze the limiting distribution of $T_N$.
By the assumed CLT for $(X_i)$,
\[
\sqrt{N}\,(\bar f_N - \mu(f)) \xrightarrow{d} N(0,\sigma_f^2).
\]
By Assumption~\ref{ass:decoupled-consistent},
\[
\widetilde{\sigma}_{f,N}^2 \xrightarrow{P} \widetilde{\sigma}_f^{\,2},
\qquad
\text{so}
\qquad
\frac{1}{\sqrt{2\,\widetilde{\sigma}_{f,N}^2}}
\;\xrightarrow{P}\;
\frac{1}{\sqrt{2\,\widetilde{\sigma}_f^{\,2}}}.
\]
Since $\widetilde{\sigma}_{f,N}^2$ is positive with probability tending to~$1$,
Slutsky's theorem yields
\[
T_N
=
\frac{\sqrt{N}(\bar f_N - \mu(f))}
     {\sqrt{2\,\widetilde{\sigma}_{f,N}^2}}
\;\xrightarrow{d}\;
N\!\left(0,\tau^2\right),
\qquad
\tau^2 := \frac{\sigma_f^2}{2\,\widetilde\sigma_f^{\,2}}.
\]
By Theorem~\ref{thm:variance}, $\tau^2 \le 1$.

Fix $\alpha\in(0,1)$.  
Let $z_{1-\alpha/2}$ denote the $(1-\alpha/2)$ standard normal quantile.
For any $\tau^2 \le 1$, the normal distribution $N(0,\tau^2)$ is stochastically dominated by $N(0,1)$ in absolute value, since
\[
\Prob(|Z_\tau| \le z_{1-\alpha/2})
  = \Prob(|\tau Z| \le z_{1-\alpha/2})
  = \Prob(|Z| \le z_{1-\alpha/2}/\tau)
  \;\ge\; \Prob(|Z| \le z_{1-\alpha/2}),
\]
where $Z\sim N(0,1)$.
Therefore,
\[
\liminf_{N\to\infty}
\Prob\!\left(|T_N|\le z_{1-\alpha/2}\right)
\;\ge\;
1-\alpha.
\]

Finally, the event $\{|T_N|\le z_{1-\alpha/2}\}$ is precisely
\[
\mu(f)
\;\in\;
\bar f_N
\pm
z_{1-\alpha/2}
\sqrt{\frac{2\,\widetilde\sigma_{f,N}^2}{N}},
\]
so
\[
\liminf_{N\to\infty}
\Prob\Bigl(
\mu(f)\in \mathrm{CI}_N(f)
\Bigr)
\;\ge\;
1-\alpha,
\]
as claimed.
\end{proof}

\begin{remark}
    Although we focus here on CLT-based confidence intervals, the variance domination result of Theorem~\ref{thm:variance} can also be combined with Chebyshev-type arguments to obtain conservative bounds without assuming asymptotic normality \cite{Rosenthal2017}. We omit these extensions, which follow standard arguments, to keep the presentation focused.
\end{remark}

\section{Case Study: The AR(1) Process}
\label{sec:ar1}

We illustrate the tangent–decoupling construction in the setting of the
stationary Gaussian autoregressive process
\begin{equation}
    X_{i+1} = a X_i + \xi_i,
    \qquad |a|<1,
    \label{eq:ar1-chain}
\end{equation}
where $(\xi_i)_{i\ge0}$ are i.i.d.\ $N(0,1)$ and independent of $X_0$.
In stationarity,
\[
    X_i \sim \mathcal N(0,\sigma_0^2),
    \qquad
    \sigma_0^2 = \frac{1}{1-a^2}.
\]
The update map is $\Phi(x,u)=ax+u$, so the auxiliary variables are the
innovations $U_i=\xi_i$.

\subsection{Tangent--decoupled sequence}
Let $(\widetilde U_i)=(\widetilde\xi_i)$ be an independent i.i.d.\ copy of
$(U_i)$.
Following Definition~\ref{def:tangent}, the tangent–decoupled chain
$(\widetilde X_i)$ associated with \eqref{eq:ar1-chain} is defined recursively by
\[
    \widetilde X_{i+1} := \Phi(X_i,\widetilde U_i)
    = a X_i + \widetilde\xi_i,
    \qquad i\ge0.
\]
Thus $\widetilde X_{i+1}$ applies the same update map to the same state $X_i$
but with an independent innovation.
Both chains have the same marginals:
\[
    \widetilde X_i \stackrel{d}{=} X_i,
    \qquad\text{for all } i.
\]

\subsection{TAVCs for linear observables}

Consider the observable $f(x)=x$.
Write
\[
    \sigma_f^2 = \lim_{N\to\infty}
        N\,\Var_\mu\!\Bigl(\frac1N\sum_{i=1}^N f(X_i)\Bigr),
    \qquad
    \widetilde\sigma_f^2 = \lim_{N\to\infty}
        N\,\Var_\mu\!\Bigl(\frac1N\sum_{i=1}^N f(\widetilde X_i)\Bigr).
\]

The original chain satisfies
\[
    \Cov(X_i,X_{i+k})=\sigma_0^2 a^k.
\]
Thus
\[
    \sigma_f^2
    = \sigma_0^2 + 2\sum_{k=1}^\infty \sigma_0^2 a^k
    = \frac{1}{(1-a)^2}.
\]

For the tangent–decoupled chain,
\[
    \widetilde X_{i+1} = a X_i + \widetilde\xi_i,
\]
and since $\widetilde\xi_i$ is independent of all other terms,
\[
    \Cov(\widetilde X_i,\widetilde X_{i+k})
    = \Cov(a X_{i-1},a X_{i+k-1})
    = \sigma_0^2 a^{k+2},
    \qquad k\ge1.
\]
Hence
\[
    \widetilde\sigma_f^2
    = \sigma_0^2 + 2\sum_{k=1}^\infty \sigma_0^2 a^{k+2}
    = \sigma_0^2\Bigl(1+\frac{2a^3}{1-a}\Bigr).
\]
Using $\sigma_0^2 = (1-a^2)^{-1}$ and simplifying gives
\[
    \widetilde\sigma_f^2
    = \frac{1-a+2a^3}{(1-a)^2(1+a)}.
\]

A direct computation shows
\[
    \frac{\widetilde\sigma_f^2}{\sigma_f^2}
    = 2\Bigl(a-\tfrac12\Bigr)^2 + \tfrac12
    \;\ge\; \tfrac12,
\]
which is consistent with the variance domination predicted by
Theorem~\ref{thm:variance}:
\[
    \sigma_f^2 \;\le\; 2\,\widetilde\sigma_f^2.
\]

The AR(1) structure allows us to go further: the tangent-decoupled chain
dominates not only integrated variances but also the entire Toeplitz covariance matrix in
Loewner order:

\begin{proposition}[Toeplitz covariance domination]\label{prop:ar1-covariance-domination}
Let $\Sigma_d$ and $\widetilde\Sigma_d$ denote the covariance matrices of 
$(X_0,\dots,X_{d-1})$ and $(\widetilde X_0,\dots,\widetilde X_{d-1})$, 
where $(X_i)$ is the stationary AR(1) chain and $(\widetilde X_i)$ is its 
tangent--decoupled version with $\widetilde X_0=X_0$.  
Then for every $d\ge1$,
\[
    \Sigma_d \;\preceq\; 2\,\widetilde\Sigma_d ,
\]
i.e.\ the difference $2\widetilde\Sigma_d - \Sigma_d$ is positive semidefinite.
\end{proposition}

\begin{proof}
From the covariance formulas computed above,
\[
    \gamma_k = \Cov(X_0,X_k)=\sigma_0^2 a^{|k|},
    \qquad
    \widetilde\gamma_k = \Cov(\widetilde X_0,\widetilde X_k)
    =
    \begin{cases}
        \sigma_0^2, & k=0,\\[2pt]
        \sigma_0^2 a^{|k|+2}, & k\neq 0,
    \end{cases}
\]
with $\widetilde X_0=X_0$ ensuring $\widetilde\gamma_0=\gamma_0$.  Their Toeplitz power spectral densities \cite{GrenanderSzego1984} are
\[
    f(\theta)=\frac{1}{1+a^2-2a\cos\theta},
    \qquad
    \widetilde f(\theta)=a^2 f(\theta)+1.
\]

To verify that $2\widetilde f(\theta)\ge f(\theta)$ for all $\theta$, observe that
\begin{equation} \label{eq:rhs_2tfmf}
     2\widetilde f(\theta)-f(\theta)
    = (2a^2-1)\,f(\theta) + 2.
\end{equation}
Since 
\[
    f(\theta)\in\Bigl[\frac{1}{(1+a)^2},\,\frac{1}{(1-a)^2}\Bigr],
\]
the minimum of the right-hand side of \eqref{eq:rhs_2tfmf} depends on the sign of $2a^2-1$.

\smallskip
\emph{Case 1: $2a^2-1 \ge 0$.}
In this case the expression is increasing in $f(\theta)$, so the minimum occurs at
$f_{\min} = (1+a)^{-2}$. Thus
\[
    \min_\theta \bigl(2\widetilde f(\theta)-f(\theta)\bigr)
    = (2a^2-1)\frac{1}{(1+a)^2} + 2 \;\ge\; 0.
\]

\smallskip
\emph{Case 2: $2a^2-1 < 0$.}
Now the expression is decreasing in $f(\theta)$, so the minimum occurs at
$f_{\max} = (1-a)^{-2}$. Hence
\[
    \min_\theta \bigl(2\widetilde f(\theta)-f(\theta)\bigr)
    = (2a^2-1)\frac{1}{(1-a)^2} + 2
    = \frac{4(a-\tfrac12)^2}{(1-a)^2} \;\ge\; 0.
\]

\smallskip

In both cases we conclude that $2\widetilde f(\theta)-f(\theta)\ge 0$ for all
$\theta\in[-\pi,\pi]$, and therefore $f(\theta)\le 2\widetilde f(\theta)$ as
claimed.

For Toeplitz covariance matrices,
\[
    v^\top \Sigma_d v
    = \frac{1}{2\pi}\int_{-\pi}^{\pi}
        |p_v(e^{i\theta})|^2 f(\theta)\,d\theta,
\]
and similarly for $\widetilde\Sigma_d$ with $\widetilde f$.  
The pointwise inequality $f\le 2\widetilde f$ therefore implies
\[
    v^\top \Sigma_d v
    \;\le\; 2\,v^\top \widetilde\Sigma_d v
    \qquad\text{for all } v\in\mathbb{R}^d,
\]
which is exactly the Loewner-order inequality
\(\Sigma_d \preceq 2\widetilde\Sigma_d\).
\end{proof}

\subsection{Quadratic and discontinuous observables}

A similar analysis can be carried out for nonlinear $f$.  
For example, for $f(x)=x^2$ one uses the Gaussian identity
\[
    \Cov(X_i^2,X_{i+k}^2)=2\,\Cov(X_i,X_{i+k})^2,
\]
yielding
\[
    \sigma_{x^2}^2
    =2\sigma_0^4\Bigl(1+\frac{2a^2}{1-a^2}\Bigr)
    =\frac{2(1+a^2)}{(1-a^2)^3},
    \qquad
    \widetilde\sigma_{x^2}^2
    =2\sigma_0^4\Bigl(1+\frac{2a^6}{1-a^2}\Bigr)
    =\frac{2(1-a^2+2a^6)}{(1-a^2)^3}.
\]
and again $\sigma_{x^2}^2 \le 2\widetilde\sigma_{x^2}^2$.

For a discontinuous observable such as
\[
    f(x)=\mathbf 1_{\{x>c\}},
\]
the covariance can be written in terms of bivariate Gaussian tail
probabilities:
\[
    \Cov\bigl(f(X_i),f(X_{i+k})\bigr)
    = \Phi_2(c,c; a^k)-\Phi(c)^2,
\]
where $\Phi$ is the standard normal CDF and
$\Phi_2(\cdot,\cdot;\rho)$ denotes the bivariate standard normal CDF
with correlation $\rho$.
For the tangent–decoupled chain this becomes
\[
    \Phi_2(c,c; a^{k+2})-\Phi(c)^2,
\]
again reducing correlations and preserving the inequality of
Theorem~\ref{thm:variance}.

\subsection{Empirical illustration}

Define the empirical means
\[
    \bar f_N := \frac1N\sum_{i=1}^N f(X_i),
    \qquad
    \bar f_N^{\,\widetilde X} := \frac1N\sum_{i=1}^N f(\widetilde X_i).
\]

Figure~\ref{fig:ar1-three-observables-consistency} illustrates the almost-sure convergence guaranteed by Theorem~\ref{thm:intro-consistency}.
For three representative observables --- $f(x)=x$, $f(x)=x^2$, and the step function
$f(x)=\mathbf 1_{\{x>c\}}$ with $c=-0.5$ --- we plot the running means
\(
\frac1n\sum_{i=1}^n f(X_i)
\)
and
\(
\frac1n\sum_{i=1}^n f(\widetilde X_i)
\)
along a single long AR(1) trajectory.  
In all cases the tangent--decoupled running mean converges to the correct limit, consistent with the theoretical behavior established in Theorem~\ref{thm:intro-consistency}.

\begin{figure}[t]
    \centering
    \includegraphics[width=0.32\textwidth]{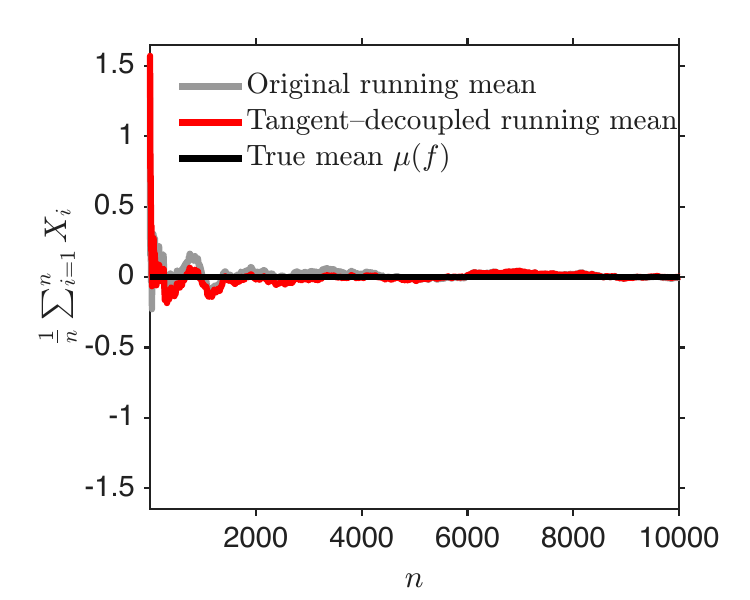}%
    \includegraphics[width=0.32\textwidth]{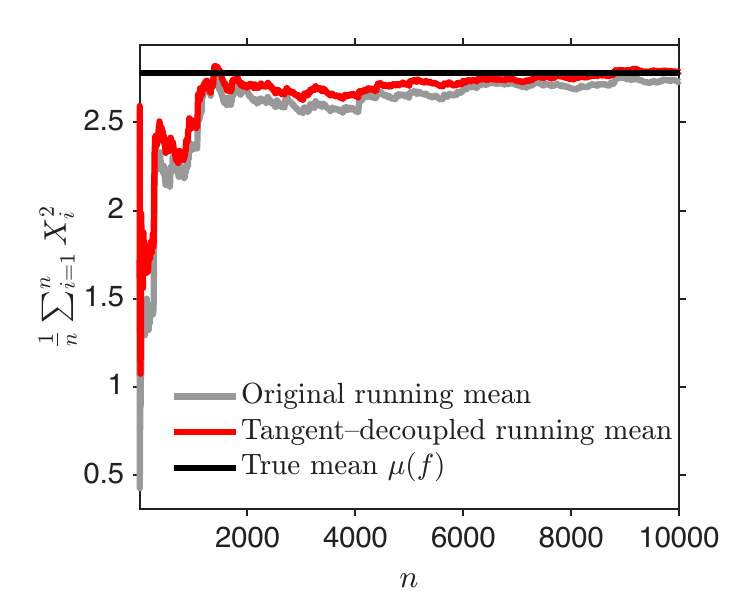}%
    \includegraphics[width=0.32\textwidth]{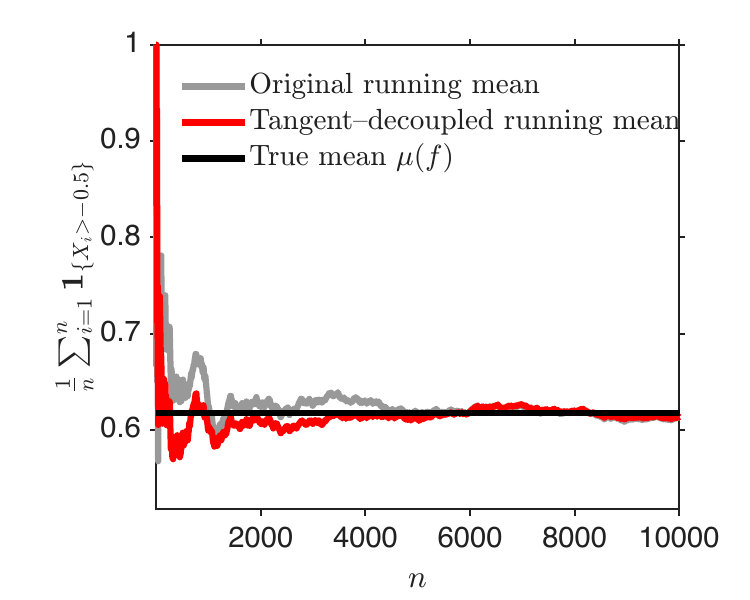}
    \caption{
    Consistency of the tangent–decoupled estimator in the AR(1) example.
    Each panel shows the running mean 
    $\frac1n\sum_{i=1}^n f(X_i)$ (solid gray) 
    and its tangent–decoupled analogue 
    $\frac1n\sum_{i=1}^n f(\widetilde X_i)$ (solid red),
    computed along a \emph{single long trajectory} of length $N=10^5$.
    The black line marks the true expectation $\mu(f)$.
    Left: $f(x)=x$.
    Middle: $f(x)=x^2$.
    Right: $f(x)=\mathbf{1}_{\{x>c\}}$ with $c=-0.5$.
    In all cases the tangent–decoupled running mean converges to the correct
    limit.
    }
    \label{fig:ar1-three-observables-consistency}
\end{figure}

Next we examine the finite-sample variance comparison stated in Theorem~\ref{thm:variance}.
Figure~\ref{fig:ar1-three-observables} reports empirical estimates of
\[
    N\,\Var(\bar f_N),
    \qquad
    2N\,\Var(\bar f_N^{\,\widetilde X}),
\]
for the same three observables.  
The inequality predicted by Theorem~\ref{thm:variance},
\[
    N\,\Var(\bar f_N)
    \;\le\;
    2N\,\Var(\bar f_N^{\,\widetilde X}),
\]
is borne out visually in all panels: the curve for $N\,\Var(\bar f_N)$ remains
below the curve for $2N\,\Var(\bar f_N^{\,\widetilde X})$, with the two
quantities typically of comparable magnitude.

\begin{figure}[t]
    \centering
    \includegraphics[width=0.32\textwidth]{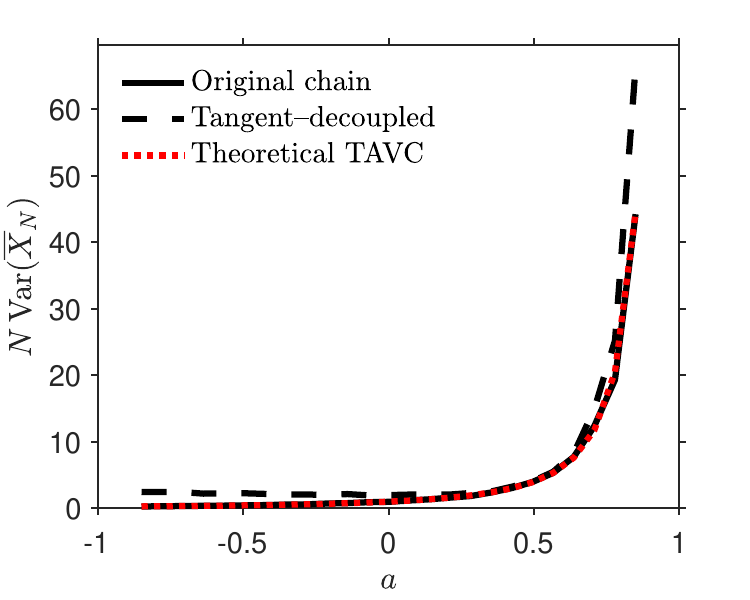}
    \includegraphics[width=0.32\textwidth]{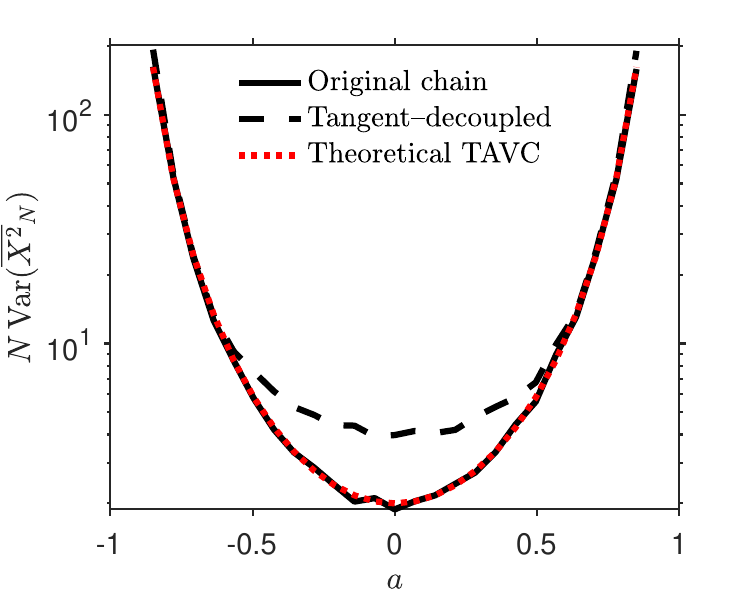}
    \includegraphics[width=0.32\textwidth]{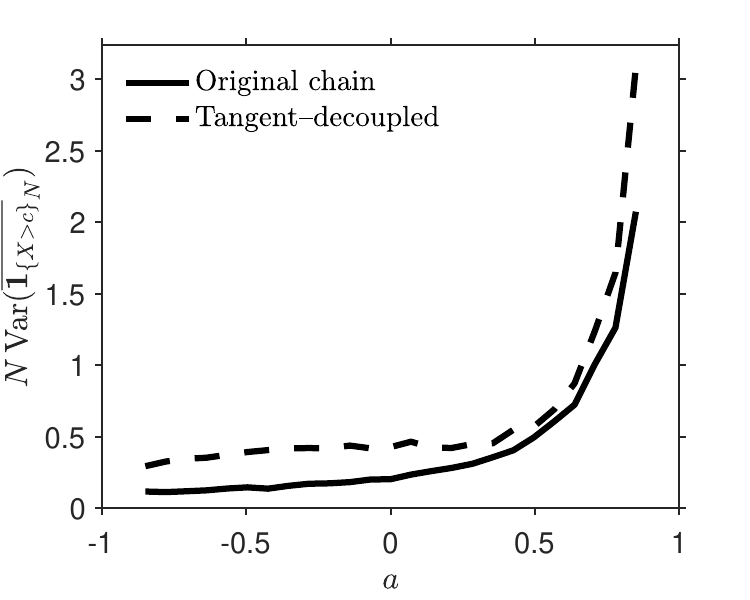}
\caption{
Comparison of the empirical time–average variance estimate 
$N\,\operatorname{Var}(\bar f_N)$ for the AR(1) chain (solid black),
its tangent–decoupled analogue 
$2N\,\operatorname{Var}(\bar f_N^{\,\widetilde X})$ (dashed black),
and the corresponding theoretical TAVC 
$\sigma_f^{\,2}$ (red dotted), for the observables 
$f(x)=x$ (left), $f(x)=x^2$ (middle, shown on a logarithmic scale), 
and the step function $f(x)=\mathbf{1}_{\{x>c\}}$ with $c=-0.5$
(right).  
In all cases the tangent–decoupled sequence exhibits weaker temporal 
dependence, and the empirical behavior clearly reflects Theorem~\ref{thm:variance}.}
\label{fig:ar1-three-observables}
\end{figure}


\section{Uncertainty Quantification in Markov Chain Monte Carlo}

\label{sec:mcmc}

 In our numerical experiments, the
tangent–decoupled sequence is generated in parallel with the main MCMC
chain.  At iteration $i$ the sampler is in state $X_i$ and produces the next
state via
\[
    X_{i+1} = \Phi(X_i, U_i),
\]
where $U_i$ denotes the auxiliary randomness used at that step (e.g.,
Bernoulli accept/reject decisions, momentum or direction refreshes, or other
algorithm-specific random choices).  The decoupled update is obtained by
applying the \emph{same} transition map to the same input state $X_i$ but
using an independent draw $U_i'$ of the auxiliary randomness:
\[
    \widetilde X_{i+1} = \Phi(X_i, U_i').
\]
The two sequences $(X_i)$ and $(\widetilde X_i)$ therefore evolve in
parallel, sharing the same backbone state at each iteration and differing
only through independent random inputs.  This construction requires no
storage beyond the current state and adds only one additional evaluation of
the transition map $\Phi$ per iteration, making both the computational and
memory overhead negligible.

Given the decoupled evaluations
\[
    f(\widetilde X_1),\dots,f(\widetilde X_N),
\]
we estimate the long--run variance $\widetilde{\sigma}_f^{\,2}$ of the
decoupled empirical mean
\[
    \bar f_N^{\,\widetilde X} = \frac{1}{N}\sum_{i=1}^N f(\widetilde X_i)
\]
using the \emph{initial positive sequence} (IPS) estimator of Geyer~\cite{geyer1992}, which underlies the effective sample size (ESS) estimators commonly used in MCMC; for an algorithmic introduction and implementation details, see the Stan Reference Manual~\cite{StanRef2.37}. 
Define centered values
\[
    e_i := f(\widetilde X_i) - \bar f_N^{\,\widetilde X},
\qquad i=1,\dots,N,
\]
and (unbiased) sample autocovariances
\[
    \widehat{\gamma}_e(k)
      := \frac{1}{N-k}\sum_{i=1}^{N-k} e_i\,e_{i+k},
    \qquad k=0,1,\dots
\]
Let $\widehat{\rho}_e(k) := \widehat{\gamma}_e(k)/\widehat{\gamma}_e(0)$
denote the empirical autocorrelations.  For $t\ge 0$ define the paired sums
\[
    \widehat P_t := \widehat{\rho}_e(2t) + \widehat{\rho}_e(2t+1).
\]
The IPS truncation index is the largest $\widehat m$ such that
$\widehat P_t>0$ for all $t=0,1,\dots,\widehat m$.  We then estimate the
integrated autocorrelation time by
\[
    \widehat\tau_{\mathrm{IPS}}
    :=
    -1 + 2\sum_{t=0}^{\widehat m} \widehat P_t,
\]
and the long--run variance by
\[
    \widetilde{\sigma}_{f,N}^2
    :=
    \widehat{\gamma}_e(0)\,\widehat\tau_{\mathrm{IPS}}.
\]
In all numerical experiments below, $\widetilde{\sigma}_{f,N}^2$ is the
variance estimator used to form confidence intervals for both the backbone
and tangent--decoupled empirical means.

The IPS estimator relies on a structural positivity property of the
autocorrelation sequence that is guaranteed for reversible Markov
chains.  Since the tangent--decoupled sequence is not itself a Markov
chain, it is not \emph{a priori} obvious that this positivity condition
continues to hold.  The following proposition shows that, provided the
backbone chain is reversible, the tangent-decoupled sequence inherits
the required positivity through an equivalent representation as a
reversible chain observed through a different test function.

\begin{proposition}[Positivity for the tangent--decoupled sequence]
\label{prop:decoupled-ips-positivity}
Assume the backbone chain $(X_i)_{i\ge 0}$ is stationary and reversible with
respect to $\pi$.  Fix $f\in L^2(\pi)$ and define the tangent--decoupled
sequence by $\widetilde X_{i+1}=\Phi(X_i,U_i')$, where $(U_i')$ are i.i.d.\
auxiliary random variables independent of $(X_i)$.  Let
$Y_i := f(\widetilde X_i) - \pi(f)$ and $\rho_Y(k)$ be the autocorrelation
function of $(Y_i)$.  Then the paired sums
\[
    P_t := \rho_Y(2t) + \rho_Y(2t+1), \qquad t\ge 0,
\]
are nonnegative:
\[
    P_t \ge 0 \quad \text{for all } t\ge 0.
\]
\end{proposition}

\begin{proof}
Define
\[
    g(x) := \mathbb{E}\!\left[f(\Phi(x,U))\right] - \pi(f),
\]
where $U$ is distributed as the auxiliary randomness in one transition.  Note that $\int g(x)\,\pi(dx)=0$, so $g\in L^2_0(\pi)$.
By conditional independence of $(U_i')$ and stationarity of $(X_i)$,
\[
    \mathrm{Cov}(Y_0,Y_k)
    = \mathrm{Cov}\!\left(g(X_0),\,g(X_k)\right)
    \qquad (k\ge 0).
\]
Since $(X_i)$ is reversible with invariant $\pi$, its Markov operator $P$ is
self-adjoint on $L^2(\pi)$, and thus admits a spectral representation: there
exists a finite nonnegative measure $\mu_g$ on $[-1,1]$ such that
\[
    \mathrm{Cov}\!\left(g(X_0),g(X_k)\right)
    = \int_{-1}^1 \lambda^k\,d\mu_g(\lambda).
\]
Normalizing by $\mathrm{Var}(g(X_0))$ yields
$\rho_Y(k)=\int_{-1}^1 \lambda^k\,d\nu(\lambda)$ for a probability measure
$\nu$ on $[-1,1]$.  Therefore,
\[
    P_t
    = \int_{-1}^1 \bigl(\lambda^{2t}+\lambda^{2t+1}\bigr)\,d\nu(\lambda)
    = \int_{-1}^1 \lambda^{2t}(1+\lambda)\,d\nu(\lambda)
    \ge 0,
\]
since $\lambda^{2t}\ge 0$ and $1+\lambda\ge 0$ for all $\lambda\in[-1,1]$.
 This
proves the claim.
\end{proof}

Figure~\ref{fig:toeplitz-nuts-sigma2-vs-m} illustrates the behavior of the
estimated long--run variance $\widehat{\sigma}^2(m)$ as a function of the
truncation parameter $m$ for three representative test functions, using
NUTS as the underlying MCMC sampler targeting a Toeplitz Gaussian
distribution with covariance
$\Sigma_{ij} = \rho^{|i-j|}$.  All experiments were run in dimension $d=100$ with correlation parameter $\rho=0.9$, a fixed leapfrog step size $h=0.125$, and a maximum tree depth of $10$. 
Across all cases, the estimates obtained from both the backbone chain and
the tangent--decoupled sequence exhibit stable behavior over a wide range
of truncation values.  In particular, no systematic drift is observed as
$m$ increases, indicating that the effective truncation selected by
Geyer’s initial positive sequence (IPS) criterion occurs well before the
maximum admissible lag.  These results support the use of the IPS
estimator for uncertainty quantification in the tangent-decoupled
setting and are consistent with the theoretical positivity properties
established above.

Taken together, these results suggest that the tangent-decoupled sequence is most useful in regimes where the cost of generating the decoupled evaluations is lower than that of running the Markov chain for an equivalent additional number of iterations. Whether this tradeoff is favorable depends on implementation details that are specific to the sampler under consideration. In many MCMC methods, however, the dominant computational cost lies in evaluating the score function, and in such cases the additional overhead of generating a tangent–decoupled sequence can be substantially smaller than doubling the chain length. This is particularly plausible for gradient-based methods such as MALA, and it also raises the possibility that parallelization strategies exploiting SIMD architectures could be effective for locally adaptive samplers, since the tangent–decoupled updates share a common backbone state at each iteration.

\begin{figure}[t]
    \centering
    \includegraphics[width=0.32\textwidth]{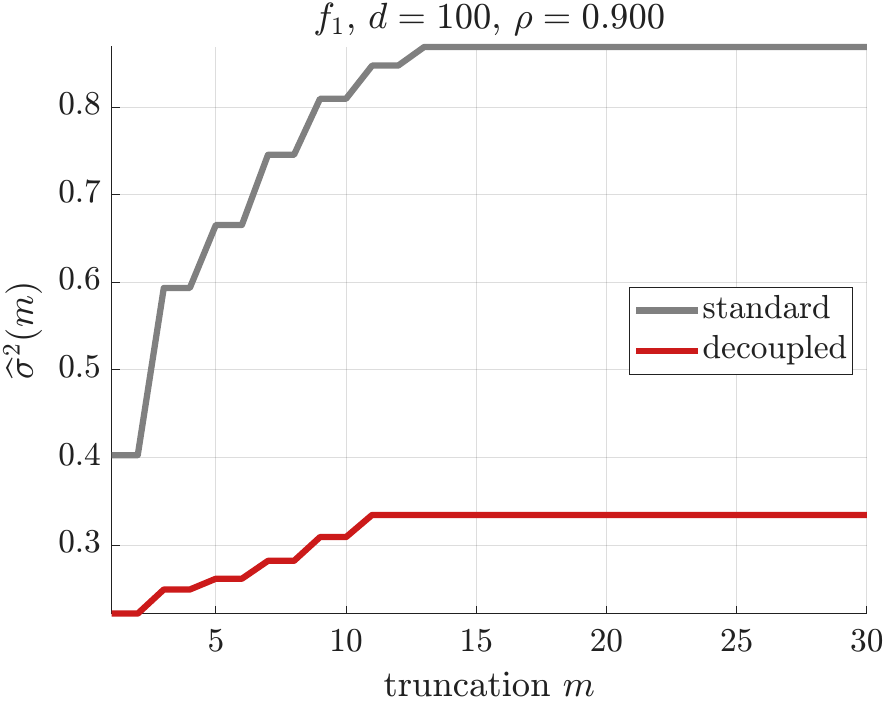}\hfill
    \includegraphics[width=0.32\textwidth]{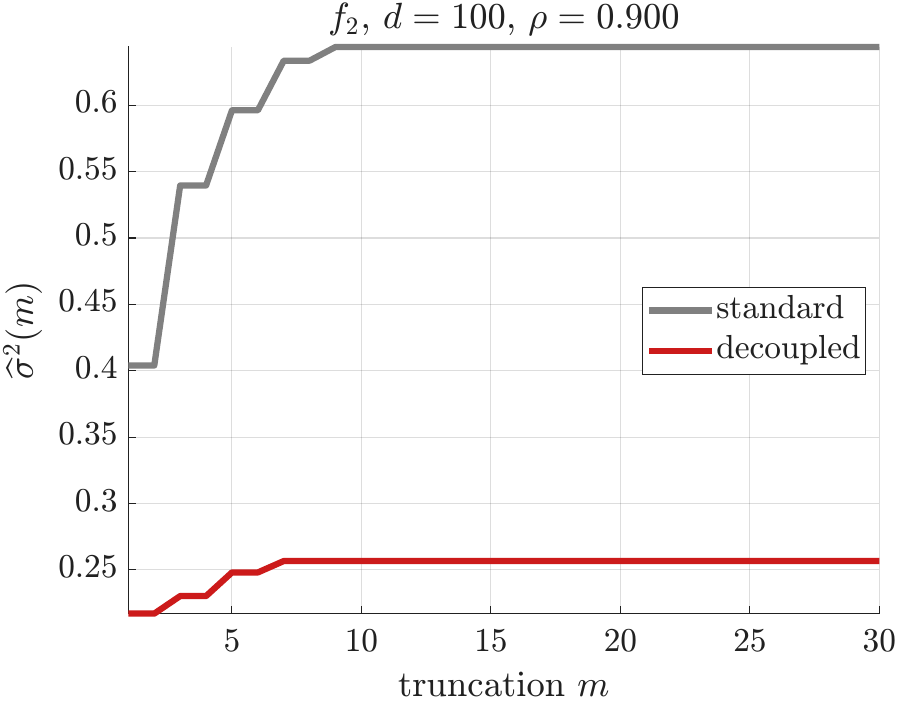}\hfill
    \includegraphics[width=0.32\textwidth]{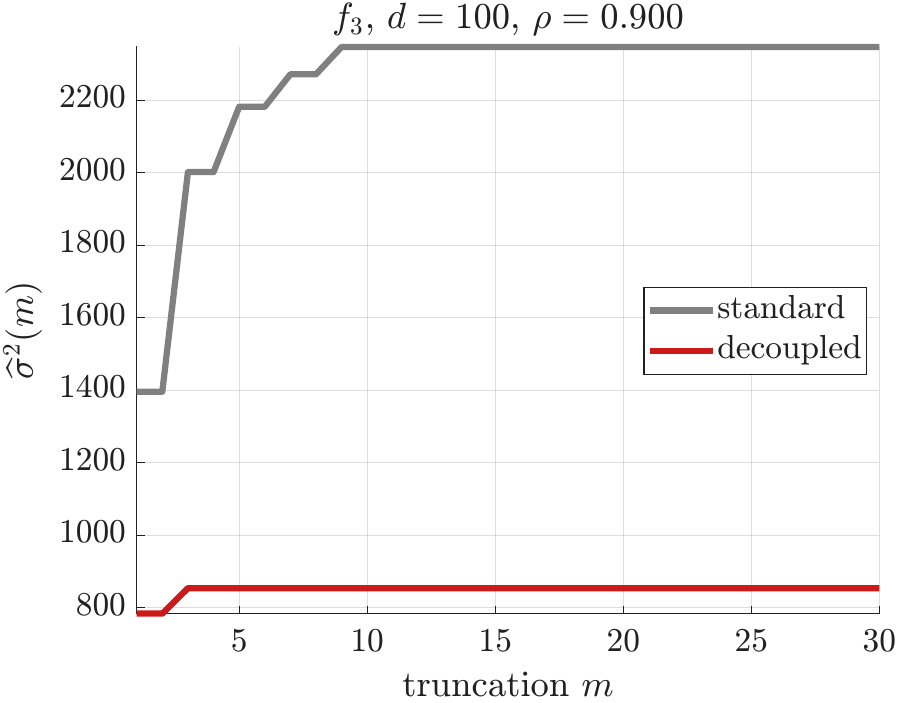}

    \caption{
    Estimated long--run variance $\widehat{\sigma}^2(m)$ as a function of the
    truncation parameter $m$ for NUTS applied to a $d=100$ Toeplitz Gaussian
    target with correlation parameter $\rho=0.9$.  Results are shown for the
    three test functions
    $f_1(\theta)=d^{-1}\sum_i \theta_i$,
    $f_2(\theta)=d^{-1}\sum_i \theta_i^2$, and
    $f_3(\theta)=\bigl(d^{-1/2}\sum_i \theta_i\bigr)^2$,
    comparing the standard backbone chain (gray) and the tangent--decoupled
    sequence (red).  Variance estimates are computed using Geyer’s initial
    positive sequence  estimator based on unbiased autocovariances.
    The empirical means of both the standard and decoupled
    sequences are close to their known true values, with the decoupled estimator exhibiting smaller relative error across all three test functions.
    }
    \label{fig:toeplitz-nuts-sigma2-vs-m}
\end{figure}


\section{Conclusion}

This paper develops a notion of tangent decoupling for Markov chains generated
via i.i.d.\ auxiliary randomness, a representation that encompasses essentially
all MCMC methods used in practice, from Gibbs samplers to Hamiltonian Monte Carlo
and NUTS. By separating the deterministic backbone of the chain from the
stochastic auxiliary inputs, we introduce a framework in which ergodic average
estimation and Monte Carlo uncertainty quantification can be carried out via a
companion process with simpler temporal structure.  This viewpoint leads to two main theoretical contributions and a practical
method for Monte Carlo error quantification.

First, we proved that the tangent–decoupled sequence
$(\widetilde X_i)$ --- obtained by running the same update map $\Phi$ along the
realized backbone $(X_i)$ but with fresh, independent auxiliary variables --- has
the correct stationary marginals and yields a \emph{consistent estimator} of
$\mu(f)$ for every integrable observable.  
Although $(\widetilde X_i)$ is neither a Markov chain nor driven by the kernel
$P$, Theorem~\ref{thm:decoupled-consistency} shows that
\[
    \frac1n \sum_{i=1}^n f(\widetilde X_i)
    \;\xrightarrow{\mathrm{a.s.}}\; \mu(f),
\]
under mild and explicit integrability assumptions.  
Thus tangent–decoupling provides a second, structurally simpler Monte Carlo
estimator with the same asymptotic mean as the original chain.

Second, leveraging the conditional independence structure of
$(\widetilde X_i)$ given the backbone, we established a \emph{uniform,
non-asymptotic variance inequality} for additive functionals of Markov chains.  
For every $f\in L^2(\mu)$ and every $N\ge1$,
\[
    \Var\!\left(\sum_{i=0}^{N-1} f(X_i)\right)
    \;\le\;
    2\,\Var\!\left(\sum_{i=0}^{N-1} f(\widetilde X_i)\right),
\]
with no reversibility, spectral conditions, or mixing assumptions required.
As an immediate corollary, the time–average variance constants satisfy
\[
    \sigma_f^2 \;\le\; 2\,\widetilde\sigma_f^{\,2}.
\]
This provides a provably conservative and easily computable upper bound for the
asymptotic variance of the original chain, enabling practical uncertainty
quantification across a wide range of MCMC algorithms.

In addition, we show that standard variance estimation procedures remain valid
for the tangent–decoupled sequence. In the reversible stationary setting, the
autocorrelation structure of the decoupled process satisfies a paired
positivity property, ensuring that the usual initial positive sequence (IPS)
and related estimators apply without modification. This guarantees that
practical variance estimators commonly used for Markov chains can be safely
applied to the tangent–decoupled sequence, strengthening its role as a
computationally tractable proxy for uncertainty quantification.

Our numerical experiments confirm the theoretical results:  
(i) the tangent–decoupled estimator converges rapidly to the correct mean, and  
(ii) the variance comparison holds sharply and stably across dimensions,
observables, and samplers, including  NUTS.

More broadly, tangent–decoupling offers a transparent way to connect a general
auxiliary-variable Markov chain to a companion process with 
weaker temporal dependence.  
The companion chain is simple enough to analyze directly while remaining
tightly linked to the original dynamics, suggesting that the approach may prove
useful beyond consistent estimators and variance bounds; for example, in deriving concentration
inequalities, studying perturbation stability, or analyzing convergence.  
We expect tangent–decoupling to have further applications in the theory and
design of modern MCMC algorithms.

\section*{Acknowledgements}
Victor de la Peña gratefully acknowledges financial support from Google DeepMind (project number GT009019).

\bibliographystyle{plain}

\end{document}